\documentclass[leqno,11pt]{amsart}
\usepackage{amsmath,amstext,amssymb,amsopn,amsthm,mathrsfs}

\allowdisplaybreaks

\DeclareMathOperator{\domain}{Dom}

\newtheorem{thm}{Theorem}[section]
\newtheorem{bigthm}{Theorem}

\newtheorem{lem}[thm]{Lemma}

\newtheorem*{rem*}{Remark}
\newtheorem{rem}[thm]{Remark}

\begin{document}

\author[A. Nowak]{Adam Nowak}
\address{Instytut Matematyczny\\
Polska Akademia Nauk\\
\'Sniadeckich 8\\
00-956 Warszawa, Poland}
\email{adam.nowak@impan.pl}

\author[L. Roncal]{Luz Roncal}
\address{Departamento de Matem\'aticas y Computaci\'on\\
Universidad de La Rioja\\
26004 Logro\~no, Spain}
\email{luz.roncal@unirioja.es}

\footnotetext{
			
		The first-named author was supported in part by MNiSW Grant N N201 417839.
		Research of the second-named author supported by the grant
			MTM2012-36732-C03-02 from Spanish Government.
}

\title[Heat kernel estimates]
	{Sharp heat kernel estimates in the Fourier-Bessel setting
	for a continuous range of the type parameter}

\begin{abstract}
The heat kernel in the setting of classical Fourier-Bessel expansions is defined by an oscillatory
series which cannot be computed explicitly. We prove qualitatively sharp estimates of this kernel.
Our method relies on establishing a connection with a situation of expansions based on Jacobi polynomials
and then transferring known sharp bounds for the related Jacobi heat kernel.
\end{abstract}

\maketitle

\section{Introduction}

Let $J_{\nu}$ denote the Bessel function of the first kind and order $\nu > -1$, and let
$\{\lambda_{n,\nu} : n \ge 1\}$ be the sequence of successive positive zeros of $J_{\nu}$.
The Fourier-Bessel heat kernel is given by the oscillating sum
\begin{equation} \label{ker_Bessel}
G_t^{\nu}(x,y) = 2(xy)^{-\nu} \sum_{n=1}^{\infty} \exp\big(-t\lambda_{n,\nu}^2\big)
	 \frac{J_{\nu}(\lambda_{n,\nu}x)J_{\nu}(\lambda_{n,\nu}y)}{|J_{\nu+1}(\lambda_{n,\nu})|^2}.
\end{equation}
The numbers $\lambda_{n,\nu}^2$ and the functions 
$u\mapsto \sqrt{2} u^{-\nu} J_{\nu}(\lambda_{n,\nu}u)/|J_{\nu+1}(\lambda_{n,\nu})|$ occurring here 
are the eigenvalues and the eigenfunctions, respectively, of the Bessel differential operator
$- \frac{d^2}{du^2} - \frac{2\nu+1}u \, \frac{d}{du}$ considered on the interval $(0,1)$.
It is well known that this kernel gives the solution of the initial-value problem for the Bessel heat
equation on $(0,1)$ with Dirichlet boundary condition at the right endpoint.

Our main result is the following qualitatively sharp description of $G_t^{\nu}(x,y)$.
\begin{bigthm} \label{thm:main}
Assume that $\nu > -1$. Given any $T>0$, there exist positive constants $C, c_1$ and $c_2$, depending
only on $\nu$ and $T$, such that
\begin{align*}
& \frac{1}C \, \big[ t + xy \big]^{-\nu-1/2} \, \frac{(1-x)(1-y)}{t+(1-x)(1-y)}
	\, \frac{1}{\sqrt{t}} \exp\bigg({- c_1 \frac{(x-y)^2}{t}}\bigg) \\
 \le & G_t^{\nu}(x,y)\\
 \le & C\, \big[ t + xy \big]^{-\nu-1/2} \, \frac{(1-x)(1-y)}{t+(1-x)(1-y)}
	\, \frac{1}{\sqrt{t}} \exp\bigg({- c_2 \frac{(x-y)^2}{t}}\bigg),
\end{align*}
uniformly in $x,y \in [0,1]$ and $0 < t \le T$. Moreover,
$$
C^{-1}\, \exp\big(-t\lambda^2_{1,\nu}\big)\, (1-x)(1-y) \le G_t^{\nu}(x,y) 
	\le C\, \exp\big(-t\lambda^2_{1,\nu}\big)\, (1-x)(1-y),
$$
uniformly in $x,y \in [0,1]$ and $t \ge T$.
\end{bigthm}

The main contents of Theorem \ref{thm:main} are of course the short time bounds. Proving them
by a direct analysis of the heavily oscillating series in \eqref{ker_Bessel} is practically impossible.
Note that the Bessel function $J_{\nu}$ is transcendental in general, and can be expressed by means
of elementary functions only if the index $\nu$ is half-integer, i.e.\ $\nu = k/2$ for some integer $k$.
Furthermore, the zeros of $J_{\nu}$ are known explicitly only when $\nu = \pm 1/2$. Note also that
the order of magnitude of the numbers $\lambda_{n,\nu}^2$
appearing in the exponential factor is $n^2$ as $n \to \infty$;
in particular, the asymptotic distribution of $\lambda_{n,\nu}^2$ is not linear.

The behavior of the Fourier-Bessel heat kernel does not seem to have been studied
before, except for our previous paper \cite{NR}. More precisely, in \cite[Theorem 3.3]{NR} we derived
qualitatively sharp estimates for $G_t^{\nu}(x,y)$ under the assumption that $\nu$ is
a half-integer not less than $-1/2$
(actually, this restriction is essential only for the short time bounds, see \cite[Theorem 3.7]{NR}).
We also conjectured analogous estimates for general $\nu > -1$, see \cite[Conjecture 4.3]{NR}, which
are now confirmed by Theorem~\ref{thm:main}. However, the methods we use in the present
paper are much different from those applied in \cite{NR}. Here the main argument is based on
a relation we establish between the Fourier-Bessel setting and the framework related
to Jacobi `functions'. This makes it possible to transfer qualitatively sharp estimates of the Jacobi
heat kernel obtained recently by Coulhon, Kerkyacharian and Petrushev \cite{CKP} and independently
by Nowak and Sj\"ogren \cite{NS}.
We remark that actually in \cite{CKP} heat kernel estimates were established in a very general setting
but, nevertheless, they do not apply directly to the Fourier-Bessel case.

Our principal motivation for investigating $G_t^{\nu}(x,y)$ comes from an interest in harmonic
analysis related to Fourier-Bessel expansions, see the references in \cite[Section 1]{NR}.
Additional motivation emerges from the probabilistic interpretation
of the Fourier-Bessel heat kernel, see \cite[Section I]{KM}.
Namely, it is a transition probability density for the
time scaled Bessel process $X^{\nu+1/2}_{2t}$ on $(0,\infty)$
(with reflecting barrier at $x=0$ when $-1< \nu <0$, see \cite[Appendix I, Section 21]{BS})
killed upon leaving the interval $(0,1)$. We also note that
for $\nu=d/2-1$, $d \ge 1$, the kernel
$G_t^{\nu}(x,y)$ provides a fundamental solution to the standard heat equation
in the $d$-dimensional Euclidean unit ball with Dirichlet boundary condition
and a radial initial condition.

There are two simple cases occurring already in Fourier's works, at least implicitly. These are
$G_t^{- 1/2}(x,y)$ and $G_t^{1/2}(x,y)$, and the two kernels
can be written by means of non-oscillating series, see \cite[Section 4]{NR}.
The argument is based on the periodized Gauss-Weierstrass kernel and simple boundary-value problems
for the classical heat equation in an interval. No more elementary representations in these cases
seem to be possible, and this suggests that a general closed explicit formula for
$G_t^{\nu}(x,y)$ does not exist. The estimates in Theorem \ref{thm:main} are therefore
a natural and desirable substitute for an exact expression.

\section{Fourier-Bessel and Jacobi settings} \label{sec:prel}

Given $\nu > -1$, consider the Bessel differential operator
$$
L^{\nu} = - \frac{d^2}{dx^2} - \frac{2\nu+1}x \, \frac{d}{dx}
$$
on the interval $(0,1)$. It is well known that $L^{\nu}$ is symmetric and nonnegative on
$C_c^2(0,1) \subset L^2(d\mu_{\nu})$, where $\mu_{\nu}$
is a power measure in the interval $[0,1]$ given by
$$
d\mu_{\nu}(x) = x^{2\nu+1} \, dx.
$$
A classical orthonormal basis in $L^2(d\mu_{\nu})$ consisting of eigenfunctions of
$L^{\nu}$ is the Fourier-Bessel system $\{\phi_n^{\nu} : n \ge 1\}$ (cf. \cite[Chapter XVIII]{Wat})
defined by
$$
\phi_n^{\nu}(x) = \frac{\sqrt{2}}{|J_{\nu+1}(\lambda_{n,\nu})|} x^{-\nu} J_{\nu}(\lambda_{n,\nu}x),
	\qquad n \ge 1, \quad x \in [0,1];
$$
here the value $\phi_n^{\nu}(0)$ is understood in the limiting sense.
We have
$$
L^{\nu} \phi_n^{\nu} = \lambda_{n,\nu}^2 \phi_n^{\nu}, \qquad n \ge 1.
$$

The self-adjoint extension of $L^{\nu}$ associated with the Fourier-Bessel system
is defined naturally in $L^2(d\mu_{\nu})$ via its spectral decomposition as
$$
\widetilde{L}^{\nu}f = \sum_{n=1}^{\infty} \lambda_{n,\nu}^2
	\langle f, \phi_n^{\nu} \rangle_{d\mu_{\nu}} \phi_n^{\nu},	
$$
on the domain
$$
\domain \widetilde{L}^{\nu} = \bigg\{ f \in L^2(d\mu_{\nu}) : \sum_{n=1}^{\infty}
	\big| \lambda_{n,\nu}^{2} \big\langle f, \phi_n^{\nu}\big\rangle_{d\mu_{\nu}} \big|^2 < \infty \bigg\}.
$$
The semigroup generated by $-\widetilde{L}^{\nu}$ has the integral representation
$$
\exp\big(-t \widetilde{L}^{\nu}\big)f(x) = \int_0^1 G_t^{\nu}(x,y)f(y)\, d\mu_{\nu}(y),
$$
where the kernel is given by \eqref{ker_Bessel}.

For technical reasons, we also need to consider a closely related setting
emerging from incorporating the measure $\mu_{\nu}$ into the system $\{\phi_n^{\nu}\}$.
In this way we derive the Fourier-Bessel system $\{\psi_n^{\nu} : n \ge 1\}$,
$$
{\psi}_n^{\nu}(x) = x^{\nu+1/2} \phi_n^{\nu}(x), \qquad n \ge 1, \quad x \in [0,1];
$$
the value $\psi_n^{\nu}(0)$ is understood in the limiting sense, in particular $\psi_n^{\nu}(0)=\infty$
for $\nu < -1/2$. For each $\nu>-1$ this system
is an orthonormal basis in $L^2(dx)$ (here and elsewhere $dx$ stands for Lebesgue
measure in the interval $[0,1]$) consisting of eigenfunctions of the differential operator
\begin{equation} \label{L_tilde}
\mathbb{L}^{\nu} = -\frac{d^2}{dx^2} - \frac{1/4-\nu^2}{x^2};
\end{equation}
more precisely, we have
$$
\mathbb{L}^{\nu} {\psi}_n^{\nu} = \lambda_{n,\nu}^2 {\psi}_n^{\nu}, \qquad n \ge 1.
$$
The operator $\mathbb{L}^{\nu}$ is symmetric and nonnegative on
$C_c^2(0,1) \subset L^2(dx)$. Similarly as in the previous setting, there is a natural self-adjoint
extension of $\mathbb{L}^{\nu}$ in $L^2(dx)$, denoted by
$\widetilde{\mathbb{L}}^{\nu}$, whose spectral decomposition is given by the ${\psi}_n^{\nu}$.
The associated heat semigroup $\exp(-t \widetilde{\mathbb{L}}^{\nu})$ possesses an integral
representation (with integration against $dx$), and the relevant kernel is
\begin{equation} \label{ker_rel}
\mathbb{K}_t^{\nu}(x,y) = (xy)^{\nu+1/2} G_t^{\nu}(x,y);
\end{equation}
this identity is understood in the limiting sense whenever $x=0$ or $y=0$.

Finally, we introduce the Jacobi setting based on Jacobi trigonometric `functions',
see \cite[Section 2]{NS}, scaled to the interval $[0,1]$. Let $P_k^{\alpha,\beta}$,
$k=0,1,2,\ldots$, be the classical Jacobi polynomials
with type parameters $\alpha,\beta>-1$, as defined in Szeg\"o's monograph \cite{Sz}. Define
$$
\Phi_k^{\alpha,\beta}(x) = c_k^{\alpha,\beta}
		\Big(\sin\frac{\pi x}2\Big)^{\alpha+1/2} \Big(\cos\frac{\pi x}2\Big)^{\beta+1/2}
		P_k^{\alpha,\beta}\big(\cos\pi x\big), \qquad k \ge 0, \quad x \in [0,1],
$$
with
$$
c_k^{\alpha,\beta} = \bigg( \frac{\pi (2k+\alpha+\beta+1)\Gamma(k+\alpha+\beta+1)\Gamma(k+1)}
	{\Gamma(k+\alpha+1)\Gamma(k+\beta+1)} \bigg)^{1/2},
$$
where for $k=0$ and $\alpha+\beta=-1$ the product $(2k+\alpha+\beta+1)\Gamma(k+\alpha+\beta+1)$
must be replaced by $\Gamma(\alpha+\beta+2)$. The values $\Phi_k^{\alpha,\beta}(0)$ and 
$\Phi_k^{\alpha,\beta}(1)$ are understood in the limiting sense and they may be infinite for some $\alpha$
and $\beta$.
Then the system $\{\Phi_k^{\alpha,\beta}: k\ge 0\}$
is an orthonormal basis in $L^2(dx)$. Moreover, each $\Phi_k^{\alpha,\beta}$ is an eigenfunction
of the differential operator
$$
\mathbb{J}^{\alpha,\beta} = -\frac{d^2}{dx^2} - \frac{\pi^2(1/4-\alpha^2)}{4\sin^2(\pi x /2)}
	- \frac{\pi^2(1/4-\beta^2)}{4\cos^2(\pi x/2)},
$$
and we have
$$
\mathbb{J}^{\alpha,\beta} \Phi_k^{\alpha,\beta} = \Lambda_k^{\alpha,\beta} \Phi_k^{\alpha,\beta},
	 \qquad k \ge 0, \qquad \textrm{where} \qquad \Lambda_k^{\alpha,\beta} =
	 \pi^2 \Big( k + \frac{\alpha+\beta+1}2\Big)^2.
$$
Thus $\mathbb{J}^{\alpha,\beta}$ has a natural self-adjoint extension in $L^2(dx)$ given by
$$
\widetilde{\mathbb{J}}^{\alpha,\beta} f = \sum_{k=0}^{\infty} \Lambda_k^{\alpha,\beta}
	\big\langle f, \Phi_k^{\alpha,\beta} \big\rangle \Phi_k^{\alpha,\beta}
$$
on the domain
$$
\domain \widetilde{\mathbb{J}}^{\alpha,\beta} = \bigg\{ f \in L^2(dx) : \sum_{k=0}^{\infty}
	\big| \Lambda_k^{\alpha,\beta} \big\langle f, \Phi_k^{\alpha,\beta}\big\rangle \big|^2 < \infty \bigg\}.
$$

The semigroup generated by $-\widetilde{\mathbb{J}}^{\alpha,\beta}$ in $L^2(dx)$
has the integral representation
$$
\exp\big(-t\widetilde{\mathbb{J}}^{\alpha,\beta}\big)f(x) =
\int_0^1 \mathbb{G}_t^{\alpha,\beta}(x,y)f(y)\, dy,
$$
where the Jacobi heat kernel is defined by the oscillating series
$$
\mathbb{G}_t^{\alpha,\beta}(x,y) = \sum_{k=0}^{\infty} \exp\big( -t\Lambda_k^{\alpha,\beta}\big)
	\Phi_k^{\alpha,\beta}(x) \Phi_k^{\alpha,\beta}(y).
$$
Recently, qualitatively sharp estimates of $\mathbb{G}_t^{\alpha,\beta}(x,y)$ were obtained in
\cite{CKP} for $\alpha,\beta > -1$ and independently in \cite{NS}, with completely different methods and
under the restriction $\alpha,\beta \ge -1/2$. The result below is a consequence of
\cite[Theorem 7.2]{CKP} (see also \cite[Theorem A]{NS} and the beginning of \cite[Section 4]{NS}),
the relation between various Jacobi heat kernels \cite[(3)]{NS} and a simple
scaling argument.
\begin{thm}[\cite{CKP,NS}] \label{thm:Jacobi}
Assume that $\alpha,\beta > -1$. Given any $T > 0$, there exist positive constants 
$C, c_1$ and $c_2$, depending only on $\alpha, \beta$ and $T$, such that
\begin{align*}
& \frac{1}C \, \bigg[ \frac{xy}{t+xy} \bigg]^{\alpha+1/2} \, 
	\bigg[\frac{(1-x)(1-y)}{t+(1-x)(1-y)}\bigg]^{\beta+1/2}
	\, \frac{1}{\sqrt{t}} \exp\bigg({- c_1 \frac{(x-y)^2}{t}}\bigg) \\
 \le & \mathbb{G}_t^{\alpha,\beta}(x,y)\\
 \le & C\, \bigg[ \frac{xy}{t+xy} \bigg]^{\alpha+1/2} \, 
	\bigg[\frac{(1-x)(1-y)}{t+(1-x)(1-y)}\bigg]^{\beta+1/2}
	\, \frac{1}{\sqrt{t}} \exp\bigg({- c_2 \frac{(x-y)^2}{t}}\bigg),
\end{align*}
uniformly in $x,y \in [0,1]$ and $0 < t \le T$, where the expressions in the lower and upper estimates
are understood in the limiting sense whenever $x$ or $y$ is an endpoint of $[0,1]$.
\end{thm}
These estimates are crucial for our reasoning proving Theorem \ref{thm:main}.

\section{Proof of Theorem \ref{thm:main}} \label{sec:proof}

We first focus on showing the short time bounds, which constitute the main part of Theorem~\ref{thm:main}.
Our argument relies on relating the Fourier-Bessel heat kernel $\mathbb{K}_t^{\nu}(x,y)$ with the Jacobi
heat kernel $\mathbb{G}_t^{\alpha,\beta}(x,y)$ and then transferring the estimates of
Theorem \ref{thm:Jacobi}. This will be achieved by proving that the generator of the
Fourier-Bessel semigroup $\{\exp(-t\widetilde{\mathbb{L}}^{\nu})\}$ is a `slight' perturbation of the
Jacobi generator with the parameters of type $\alpha = \nu$ and $\beta =1/2$.
Then the desired relation will follow by the so-called Trotter product formula.

A close connection between the generators is suggested by a relation at the level of differential
operators. Observe that ${\mathbb{L}}^{\nu}$ and $\mathbb{J}^{\nu,1/2}$ differ only by a zero
order term,
$$
\mathbb{L}^{\nu} - \mathbb{J}^{\nu,1/2} = \bigg(\frac{1}4-\nu^2\bigg)
	\Bigg[ \frac{\pi^2}{4\sin^2 \frac{\pi x}2} - \frac{1}{x^2}\Bigg] =: H^{\nu}(x).
$$
Moreover, the difference function $H^{\nu}(x)$ is continuous and bounded on $[0,1]$, with the value at $x=0$
understood in the limiting sense. In the sequel we will use repeatedly, sometimes implicitly, the bound
$$
|H^{\nu}(x)| \le |H^{\nu}(1)| = \bigg| \frac{1}4-\nu^2\bigg| \bigg(\frac{\pi^2}4 -1\bigg),
\qquad x \in [0,1].
$$
This follows by analyzing the function in square brackets above. Indeed, its 
limit value at $x=0$ is positive, and its derivative is positive in
$(0,1)$, as easily verified by means of the inequality
$$
\sin^3 u > u^3 \cos u, \qquad 0 < u < \frac{\pi}2.
$$
The latter can readily be checked by using the estimates
$$
\sin u > u - \frac{u^3}{6}, \qquad \cos u < 1-\frac{u^2}2 + \frac{u^4}{24}, \qquad u > 0,
$$
coming from Taylor expansions of the two trigonometric functions.

An essential point of the whole method is the coincidence of domains, which is stated in the following.
\begin{thm} \label{thm:dom}
For each $\nu > -1$,
$$
\domain \widetilde{\mathbb{L}}^{\nu} = \domain \widetilde{\mathbb{J}}^{\nu,1/2}.
$$
\end{thm}

To prove this we will need a preparatory result.
\begin{lem} \label{lem:dom}
Let $\nu > -1$. Then
\begin{equation*}
\big\langle \widetilde{\mathbb{L}}^{\nu} {\psi}_n^{\nu}, \Phi_k^{\nu,1/2}\big\rangle =
\big\langle {\psi}_n^{\nu}, \mathbb{L}^{\nu}\Phi_k^{\nu,1/2}\big\rangle, \qquad
	n \ge 1, \quad k \ge 0,
\end{equation*}
where $\mathbb{L}^{\nu}$ is the differential operator given by \eqref{L_tilde}.
\end{lem}

\begin{proof}
We shall use the divergence form of $\mathbb{L}^{\nu}$,
$$
\mathbb{L}^{\nu} f(x) =
	- x^{-\nu-1/2} D \Big( x^{2\nu+1} D \big( x^{-\nu-1/2}f(x)\big)\Big),
$$
and integrate by parts. Denote
$$
\mathcal{I} = \big\langle \widetilde{\mathbb{L}}^{\nu} {\psi}_n^{\nu}, \Phi_k^{\nu,1/2}\big\rangle
	= \int_0^1 \mathbb{L}^{\nu} {\psi}_n^{\nu}(x) \Phi_k^{\nu,1/2}(x)\, dx.
$$
Clearly, the integral here converges since $\mathbb{L}^{\nu} {\psi}_n^{\nu}$
and $\Phi_k^{\nu,1/2}$ are in $L^2(dx)$.
We have
\begin{align*}
\mathcal{I} & = - \int_0^1 x^{-\nu-1/2} D \Big( x^{2\nu+1} D
	\big( x^{-\nu-1/2}{\psi}_n^{\nu}(x)\big)\Big)
	\Phi_k^{\nu,1/2}(x)\, dx \\
& = -x^{\nu+1/2} D\big(x^{-\nu-1/2}{\psi}_n^{\nu}(x)\big) \Phi_k^{\nu,1/2}(x)\Big|_0^1
	+ \int_0^1 D\big( x^{-\nu-1/2}{\psi}_n^{\nu}(x)\big)
		D\big( x^{-\nu-1/2}\Phi_k^{\nu,1/2}(x)\big)\, x^{2\nu+1}\, dx\\
& = -x^{\nu+1/2} D\big(x^{-\nu-1/2}{\psi}_n^{\nu}(x)\big) \Phi_k^{\nu,1/2}(x)\Big|_0^1
	+ x^{\nu+1/2} {\psi}_n^{\nu}(x) D\big( x^{-\nu-1/2}\Phi_k^{\nu,1/2}(x)\big)\Big|_0^1 \\
& \quad - \int_0^1 {\psi}_n^{\nu}(x)
	x^{-\nu-1/2} D \Big( x^{2\nu+1} D \big( x^{-\nu-1/2}\Phi_k^{\nu,1/2}(x)\big)\Big)\, dx\\
&	\equiv \mathcal{I}_1 + \mathcal{I}_2 + \mathcal{I}_3.
\end{align*}
Since $\mathcal{I}_3 = \langle {\psi}_n^{\nu}, \mathbb{L}^{\nu}\Phi_k^{\nu,1/2}\rangle$,
it remains to ensure that $\mathcal{I}_1 = \mathcal{I}_2 = 0$.

From the definition of $\Phi_{k}^{\alpha,\beta}$ it is clear that
$$
\Phi_k^{\nu,1/2}(x) = \mathcal{O}(x^{\nu+1/2}), \quad x \to 0^+, \qquad \textrm{and} \qquad
\Phi_k^{\nu,1/2}(x) = \mathcal{O}(1-x), \quad x \to 1^-.
$$
The estimate $|\phi_n^{\nu}(x)| \le C(1-x)$, see the proof of \cite[Theorem 3.7]{NR},
holds uniformly in $x\in (0,1)$ and implies
$$
{\psi}_n^{\nu}(x) = \mathcal{O}(x^{\nu+1/2}), \quad x \to 0^+, \qquad \textrm{and} \qquad
{\psi}_n^{\nu}(x) = \mathcal{O}(1-x), \quad x \to 1^-.
$$
Further, using $D(z^{-\nu}J_{\nu}(z)) = -z^{-\nu}J_{\nu+1}(z)$
(cf. \cite[Chapter III, Section 3$\cdot$2]{Wat}), we see that
$$
D\big( x^{-\nu-1/2}{\psi}_n^{\nu}(x)\big) = c\, x^{-\nu} J_{\nu+1}(\lambda_{n,\nu}x) =
	\begin{cases}
		\mathcal{O}(x), \quad x \to 0^+ \\
		\mathcal{O}(1), \quad x \to 1^-
	\end{cases},
$$
where $c$ does not depend on $x$. Moreover, by the differentiation rule for Jacobi polynomials
(cf. \cite[(4.21.7)]{Sz}),
$$
\frac{d}{du} P_k^{\alpha,\beta}(u) = \frac{1}{2} (k+\alpha+\beta+1) P_{k-1}^{\alpha+1,\beta+1}(u),
\qquad k \ge 0,
$$
(here we use the convention that $P_{-1}^{\alpha,\beta} = 0$) it follows that
\begin{align*}
D\big( x^{-\nu-1/2}\Phi_k^{\nu,1/2}(x)\big) & = c_1 \sin(\pi x) P_{k-1}^{\nu+1,3/2}\big(\cos(\pi x)\big)
	\bigg( \frac{\sin\frac{\pi x}2}{x}\bigg)^{\nu+1/2} \cos\frac{\pi x}2\\
& \quad + c_2 P_{k}^{\nu,1/2}\big(\cos(\pi x)\big)
	\bigg( \frac{\sin\frac{\pi x}2}{x}\bigg)^{\nu+1/2} \sin\frac{\pi x}2 \\
& \quad + c_3 P_{k}^{\nu,1/2}\big(\cos(\pi x)\big)
	\frac{\frac{\pi x}2 \cos\frac{\pi x}2 - \sin\frac{\pi x}2}{x^2}
	\bigg( \frac{\sin\frac{\pi x}2}{x}\bigg)^{\nu-1/2} \cos\frac{\pi x}2,
\end{align*}
with $c_1,c_2,c_3$ independent of $x$. Therefore
$$
D\big( x^{-\nu-1/2}\Phi_k^{\nu,1/2}(x)\big) =
	\begin{cases}
		\mathcal{O}(x), \quad x \to 0^+ \\
		\mathcal{O}(1), \quad x \to 1^-
	\end{cases}.
$$
Altogether, these facts imply $\mathcal{I}_1 = \mathcal{I}_2 = 0$.
The lemma follows.
\end{proof}

\begin{proof}[Proof of Theorem \ref{thm:dom}]
We demonstrate that $\domain \widetilde{\mathbb{J}}^{\nu,1/2} \subset \domain \widetilde{\mathbb{L}}^{\nu}$.
Then the other inclusion follows by self-adjointness of the two operators.

Assume that $f \in \domain \widetilde{\mathbb{J}}^{\nu,1/2}$ and let
$$
\mathcal{S} := \sum_{n=1}^{\infty}
	 \big|\lambda_{n,\nu}^{2}\big\langle f, \psi_n^{\nu}\big\rangle\big|^2.
$$
We must show that $\mathcal{S}$ is finite. We have
$$
\mathcal{S} = \sum_{n=1}^{\infty}
	 \big|\big\langle f, \widetilde{\mathbb{L}}^{\nu}\psi_n^{\nu}\big\rangle\big|^2 =
\sum_{n=1}^{\infty} \bigg| \sum_{k=0}^{\infty} \big\langle f, {\Phi}_k^{\nu,1/2}\big\rangle
	\big\langle{\Phi}_k^{\nu,1/2},\widetilde{\mathbb{L}}^{\nu} \psi_n^{\nu} \big\rangle \bigg|^2.
$$
Applying now Lemma \ref{lem:dom} we can write
\begin{align*}
\mathcal{S} & =  \sum_{n=1}^{\infty} \bigg| \sum_{k=0}^{\infty}
	\big\langle f, {\Phi}_k^{\nu,1/2}\big\rangle
	\big\langle (\mathbb{J}^{\nu,1/2}+H^{\nu}) {\Phi}_k^{\nu,1/2},\psi_n^{\nu} \big\rangle \bigg|^2 \\
& = \sum_{n=1}^{\infty} \bigg| \bigg\langle \sum_{k=0}^{\infty} \Lambda_{k}^{\nu,1/2}
	\big\langle f, {\Phi}_k^{\nu,1/2}\big\rangle {\Phi}_k^{\nu,1/2}, \psi_n^{\nu}\bigg\rangle
	+ \bigg\langle \sum_{k=0}^{\infty}
	\big\langle f, {\Phi}_k^{\nu,1/2}\big\rangle {\Phi}_k^{\nu,1/2},
	 \psi_n^{\nu} H^{\nu}\bigg\rangle \bigg|^2 \\
& = \sum_{n=1}^{\infty} \Big| \big\langle \widetilde{\mathbb{J}}^{\nu,1/2}f, \psi_n^{\nu}\big\rangle
	+ \big\langle f H^{\nu}, \psi_n^{\nu} \big\rangle \Big|^2.
\end{align*}
This together with Parseval's identity implies
$$
\mathcal{S} \le 2 \big\| \widetilde{\mathbb{J}}^{\nu,1/2}f\big\|_{L^2(dx)}^2
	+ 2\|f H^{\nu}\|_{L^2(dx)}^2 < \infty
$$
and consequently $f \in \domain \widetilde{\mathbb{L}}^{\nu}$.
\end{proof}

We now recall the Trotter product formula, see \cite[Chapter VIII, Section 8]{RS}.
Let $A$ and $B$ be (possibly unbounded) self-adjoint operators on a Hilbert space $\mathcal{H}$.
If $A+B$ is essentially self-adjoint on $\domain A \cap \domain B$, and $A$ and $B$ are bounded from below,
then
$$
\exp\big( -t(A+B)\big) h = \lim_{m\to \infty} \Big( \exp(-tA/m) \exp(-tB/m)\Big)^m h,
	\qquad h \in \mathcal{H}, \quad t \ge 0.
$$

Specifying this general situation to our context we take $\mathcal{H} = L^2(dx)$,
$A = \widetilde{\mathbb{J}}^{\nu,1/2}$ and choose $B$ to be the multiplication operator by $H^{\nu}$.
Then, in view of Theorem \ref{thm:dom}, $A+B = \widetilde{\mathbb{L}}^{\nu}$.
Moreover, the Trotter product formula
applies and since $|H^{\nu}(x)| \le |H^{\nu}(1)|$, $x \in [0,1]$, it implies
$$
e^{-t|H^{\nu}(1)|} \exp\big(-t\widetilde{\mathbb{J}}^{\nu,1/2} \big)f
	\le \exp\big(-t\widetilde{\mathbb{L}}^{\nu}\big)f \le
	e^{t|H^{\nu}(1)|} \exp\big(-t\widetilde{\mathbb{J}}^{\nu,1/2} \big)f, \qquad 0 \le f \in L^2(dx),
$$
for each $t \ge 0$. Taking now into account that
the Jacobi and the Fourier-Bessel heat kernels are continuous functions of their arguments
(see \cite[Section 2]{NS}, \cite[Section 2]{NR}), we infer that
$$
e^{-t|H^{\nu}(1)|} \mathbb{G}_t^{\nu,1/2}(x,y) \le \mathbb{K}_t^{\nu}(x,y) \le
	e^{t|H^{\nu}(1)|} \mathbb{G}_t^{\nu,1/2}(x,y), \qquad x,y \in [0,1], \quad t >0.
$$
This combined with Theorem \ref{thm:Jacobi} and \eqref{ker_rel}
justifies the short time bounds stated in Theorem \ref{thm:main}.

The long time estimates were proved in \cite[Theorem 3.7]{NR}, under the assumption
that the threshold $T$ is sufficiently large. However, given any $0<T_0<T$, the already proved
short time bounds imply existence of a constant $C>0$ such that
$$
C^{-1}\, (1-x)(1-y) \le G_t^{\nu}(x,y) \le C\, (1-x)(1-y), \qquad x,y \in [0,1], \quad T_0 \le t \le T,
$$
and the desired conclusion follows.

The proof of Theorem \ref{thm:main} is complete.

\begin{rem}
A slightly more detailed analysis reveals that for
$\nu \in [-1/2,1/2]$
$$
e^{-tH^{\nu}(1)} \mathbb{G}_t^{\nu,1/2}(x,y) \le \mathbb{K}_t^{\nu}(x,y) \le
	e^{-tH^{\nu}(0)} \mathbb{G}_t^{\nu,1/2}(x,y) \le \mathbb{G}_t^{\nu,1/2}(x,y),
	\qquad x,y \in [0,1], \quad t >0,
$$
and for $\nu \notin [-1/2,1/2]$
$$
\mathbb{G}_t^{\nu,1/2}(x,y) \le
e^{-tH^{\nu}(0)} \mathbb{G}_t^{\nu,1/2}(x,y) \le \mathbb{K}_t^{\nu}(x,y) \le
	e^{-tH^{\nu}(1)} \mathbb{G}_t^{\nu,1/2}(x,y), \qquad x,y \in [0,1], \quad t >0.
$$
Notice that $\mathbb{K}_t^{\pm 1/2}(x,y) = \mathbb{G}_t^{\pm 1/2, 1/2}(x,y)$.
\end{rem}


\end{document}